\documentclass{article}

\usepackage{amsfonts}
\usepackage{amsmath}
\usepackage{amssymb,amsthm}
\usepackage{fixltx2e}
\usepackage{geometry}

\MakeRobust{\overrightarrow}

\newtheorem{theorem}{Theorem}

\newtheorem{corollary}[theorem]{Corollary}

\newtheorem{lemma}[theorem]{Lemma}

\newtheorem{proposition}[theorem]{Proposition}
\newtheorem{remark}[theorem]{Remark}

\begin{document}

\title{Two characterizations of simple circulant tournaments}
\author{Bernardo Llano\footnote{Research supported by
CONACYT Project CB-2012-01 178910.} \\
Departamento de Matem\'{a}ticas \\ Universidad Aut\'{o}noma Metropolitana, Iztapalapa  \\  Mexico}
\maketitle

\begin{abstract}

The \textit{acyclic disconnection} $\overrightarrow{\omega }(D)$
(resp. the
\textit{directed triangle free disconnection } $\overrightarrow{\omega }%
_{3}(D)$) of a digraph $D$ is defined as the maximum possible number
of connected components of the underlying graph of $D\setminus
A(D^{\ast })$ where $D^{\ast }$ is an acyclic (resp. a directed
triangle free) subdigraph
of $D$. We say that $\pi $ is an \textit{externally acyclic }(resp. a\textit{%
\ }$\overrightarrow{C}_{3}$\textit{-free})\textit{\ partition }if
the set of
external arcs of $D$ (an arc $(u,v)$ of $D$ is said to be \textit{external }%
if\textit{\ }$u$ and $v$ belong to different equivalence classes of
$\pi $) induces an acyclic subdigraph (resp. a
$\overrightarrow{C}_{3}$-free
subdigraph) of $D.$ A digraph $D$ is said to be $\overrightarrow{\omega }$%
\textit{-keen }(resp. $\overrightarrow{\omega
}_{3}$\textit{-keen})\textit{\ }if there exists an optimal
externally acyclic (resp. $\overrightarrow{C}_{3} $-free) partition
$\pi $ of $V(D)$ having exactly one equivalence class of
cardinality one and no externally acyclic (resp. $\overrightarrow{C}_{3}$%
-free) partition leaves more than one such a class. In this paper,
we generalize some previous results and solve some problems
posed by V. Neumann-Lara (The acyclic disconnection of a digraph,
Discrete Math. 197/198 (1999), 617-632). Let
$\overrightarrow{C}_{2n+1}(J)$ be a circulant
tournament. We prove that $\overrightarrow{C}_{2n+1}(J)$ is $\overrightarrow{%
\omega }$-keen and $\overrightarrow{\omega _{3}}$-keen, respectively, and $%
\overrightarrow{\omega }(\overrightarrow{C}_{2n+1}(J))=\overrightarrow{%
\omega }_{3}(\overrightarrow{C}_{2n+1}(J))=2$ for every $\overrightarrow{C}%
_{2n+1}(J)$. Finally, it is showed that $\overrightarrow{\omega }_{3}(%
\overrightarrow{C}_{2n+1}(J))=2$,  $\overrightarrow{C}_{2n+1}(J)$ is
simple and $J$ is aperiodic are equivalent propositions. The proofs
of the results are essentially based on classic results for abelian
groups of additive number theory.


\noindent  \emph{Math. Subj. Class.: 05C20, 11P70}

\end{abstract}

\section{ Introduction}

The \textit{acyclic disconnection} $\overrightarrow{\omega }(D)$ of a
digraph $D$ is defined in \cite{VNL} as the maximum possible number of
connected components of the underlying graph of $D\setminus A(D^{\ast })$
where $A(D^{\ast })$ denotes the arc set of an acyclic subdigraph $D^{\ast }$
of $D$ (that is, $D^{\ast }$ contains no directed cycle). The so called
\textit{directed triangle-free disconnection }(in brief, the $%
\overrightarrow{C}_{3}$\textit{-free disconnection}) $\overrightarrow{\omega
}_{3}(D)$ of $D$ (a closed definition introduced\textit{\ } in the same
paper) is the maximum possible number of connected components of the
underlying graph of $D\setminus A(D^{\ast })$ where $D^{\ast }$ is a
directed triangle-free subdigraph of $D$ (that is, $D^{\ast }$ contains no $%
\overrightarrow{C}_{3}$). Equivalently, $\overrightarrow{\omega }(D)$ can be
defined as the maximum number of colors in a coloring of the vertices of $D$
such that no cycle is properly colored (in a proper coloring, consecutive
vertices of the directed cycle receive different colors). Similarly, the $%
\overrightarrow{C}_{3}$-\textit{free} \textit{disconnection} $%
\overrightarrow{\omega }_{3}(D)$ of $D$ is the maximum number of colors in a
coloring of the vertices of $D$ such that no $3$-cycle $\overrightarrow{C}%
_{3}$ is $3$-colored. A small value of both parameters implies a large
cyclic complexity of the digraph $D.$ Apart from the first and seminal
results for these parameters in general digraphs that were developed in \cite%
{VNL}, the acyclic and the $\overrightarrow{C}_{3}$-free disconnection have
mainly been studied for tournaments, specially, for circulant tournaments
(see for instance the already mentioned paper, \cite{GN} and \cite{Ll-NL})
and for some classes of (non-circulant) regular tournaments in \cite{Ll-O}.

The following definitions were introduced in \cite{VNL}. Let $k\in \mathbb{N}
$. Consider a partition $\pi =V_{1}\mid V_{2}\mid \cdots \mid V_{k}$  of the
vertex set of a digraph $D$ into nonempty subsets $V_{j}$ for $1\leq j\leq
k. $ An equivalence class $V_{j}$ of $\pi $ is \textit{singular }if\textit{\
}$V_{j}$ is a singleton. An arc $(u,v)\in A(D)$ is said to be \textit{%
external }if\textit{\ }$u$ and $v$ belong to different equivalence classes
of $\pi .$ We say that $\pi $ is an \textit{externally acyclic }(resp. a%
\textit{\ }$\overrightarrow{C}_{3}$\textit{-free})\textit{\ partition }if
the set of external arcs of $D$ induces an acyclic subdigraph (resp. a $%
\overrightarrow{C}_{3}$-free subdigraph) of $D$. An externally acyclic
(resp. $\overrightarrow{C}_{3}$-free) partition $\pi =V_{1}\mid \cdots \mid
V_{k}$ of $V(D)$ of a digraph $D$ is \textit{optimal} if $\overrightarrow{%
\omega }(D)=k$ (resp. $\overrightarrow{\omega }_{3}(D)=k$) . A digraph $D$
is said to be $\overrightarrow{\omega }$\textit{-keen }(resp. $%
\overrightarrow{\omega }_{3}$\textit{-keen})\textit{\ }if every optimal
externally acyclic (resp. $\overrightarrow{C}_{3}$-free) partition $\pi $ of
$V(D)$ has exactly one singular equivalence class.

In \cite{Ll-NL}, it was proved that every prime order circulant tournament $%
\overrightarrow{C}_{p}(J)$ with $p\geq 3$ (see the definition in the next
section) is $\overrightarrow{\omega }_{3}$-keen, $\overrightarrow{\omega }$%
-keen and $\overrightarrow{\omega }(\overrightarrow{C}_{p}(J))=$ $%
\overrightarrow{\omega }_{3}(\overrightarrow{C}_{p}(J))=2$.

In this paper, we generalize the aforementioned results and solve the
following problems posed in \cite{VNL}:

\begin{enumerate}
\item[(i)] On p. 623, end of Section 3, it is said: "We do not know whether
every circulant tournament is $\overrightarrow{\omega }$-keen (resp. $%
\overrightarrow{\omega }_{3}$-keen)." We prove that every circulant
tournament is $\overrightarrow{\omega }$-keen and $\overrightarrow{\omega }%
_{3}$-keen, see Theorems \ref{keen-w} and \ref{keen-w3}, respectively.

\item[(ii)] Problem 6.6.1. Characterize the circulant (or regular)
tournaments $T$ for which $\overrightarrow{\omega }(T)=2$ (resp. $%
\overrightarrow{\omega }_{3}(T)=2$). Are they the same? We prove that
\begin{equation*}
\overrightarrow{\omega }(\overrightarrow{C}_{2n+1}(J))=\overrightarrow{%
\omega }_{3}(\overrightarrow{C}_{2n+1}(J))=2
\end{equation*}
for every simple circulant tournaments $\overrightarrow{C}_{2n+1}(J)$ (a
consequence of Theorem \ref{circ-tight} and Corollary \ref{final}).

\item[(iii)] Problem 6.6.3: $\overrightarrow{\omega }(\overrightarrow{C}%
_{2n+1}(J))=$ $\overrightarrow{\omega }_{3}(\overrightarrow{C}_{2n+1}(J))$
for every $\overrightarrow{C}_{2n+1}(J)$ circulant tournament (Corollary \ref%
{final}). This problem remains open when $T$ is regular in general.
\end{enumerate}

We give an interesting characterization of compositions of circulant
tournaments in terms of an additive combinatorial property of their
symbol sets (Theorem \ref{comp-quasi}) and in terms of their
$\overrightarrow{C}_{3} $-free disconnection (Theorem
\ref{circ-tight}, a fact conjectured by V. Neumann-Lara in a
personal communication). The main results of these paper are
summarized in Corollary \ref{char}. The proofs of the theorems are
essentially based on classic results for abelian groups of
additive number theory. Additive (or combinatorial as P. Erd\H{o}s used to call it) number
theory could be described as the study of the discrete properties of
subsets of (abelian) groups in general. In this work, we focus on
applying these tools to the group of residues modulo a positive integer $n$ in order to
solve three coloring and structural problems for tournaments.

\section{Preliminary results}

A tournament $T$ on $n$ vertices is an orientation of the complete graph $%
K_{n}$. For every tournament $T,$ we have that $\overrightarrow{\omega }%
_{3}(T)\geq 2$ (\cite{VNL}, Remark 4.1.(i)). We say that $T$ is \textit{tight%
} if $\overrightarrow{\omega }_{3}(T)=2.$

\begin{proposition}[\protect\cite{VNL}, Corollary 2.6]
If $T$ is a regular tournament, then every externally $\overrightarrow{C}%
_{3} $-free partition of $T$ has at most one singular equivalence class.%
\label{prop1}
\end{proposition}

\begin{remark}
(i) As a consequence of this proposition, if a regular tournament $T$ has an
optimal externally $\overrightarrow{C}_{3}$-free partition with a unique
singular equivalence class, then $T$ is $\overrightarrow{\omega }_{3}$-keen.

(ii) In particular, if $\overrightarrow{\omega }_{3}(T)=2$ (that is, $T$ is
tight), where $T$ is regular, then $T$ is $\overrightarrow{\omega }_{3}$%
-keen.\label{obs1}
\end{remark}

Throughout this paper, we deal with regular tournaments which are well-known
to have odd order. We denote by $\mathbb{Z}_{2m+1}$ the group of residues
modulo $2m+1,$ where $m$ is a positive integer. Let $J\subseteq \mathbb{Z}%
_{2m+1}\setminus \{0\}$ such that $\left\vert J\right\vert =m$ and $%
\left\vert \{j,-j\}\cap J\right\vert =1$ for every $j\in J$ (or
equivalently, $\overline{J}=-J\cup \{0\}).$ A \textit{circulant tournament }$%
T=\overrightarrow{C}_{2m+1}(J)$ (or a rotational tournament $RT_{2m+1}(J),$ $%
m\in \mathbb{N}$, see \cite{Reid}) with \textit{symbol set }$J$ is defined
by
\begin{eqnarray*}
V(\overrightarrow{C}_{2m+1}(J)) &=&\mathbb{Z}_{2m+1}\text{ and} \\
A(\overrightarrow{C}_{2m+1}(J)) &=&\left\{ (i,j):i,j\in \mathbb{Z}_{2m+1}%
\text{ and }j-i\in J\right\} .
\end{eqnarray*}
Notice that from the definition, $0\notin J+J$ and $\left\vert -J\right\vert
=m.$ It is well-known that a circulant tournament is regular and
vertex-transitive, that is for every pair of vertices $u,v\in V(%
\overrightarrow{C}_{2m+1}(J))$, there exists an automorphism of $%
\overrightarrow{C}_{2m+1}(J)$ that maps $u$ to $v$. This property is constantly used in proofs as it shows that circulant tournaments are highly symmetric structures.

The tournament $\overrightarrow{C}_{2m+1}(1,2,...,m)$ is called the \textit{%
cyclic tournament. }It is easy to check that for all $m$ there is only one
cyclic tournament up to isomorphism.

Let $D$ and $F$ be digraphs and $F_{v}$ a family of mutually disjoint
isomorphic copies of $F$ for all $v\in V(D).$ The \textit{composition }(or
\textit{lexicographic product}, see \cite{K-I})\textit{\ }$D[F]$ of the
digraphs $D$ and $F$ is defined by
\begin{eqnarray*}
V(D[F]) &=&\bigcup_{v\in V(D)}V(F_{v})\text{ and} \\
A(D[F]) &=&\bigcup_{v\in V(D)}A(F_{v})\cup \{(i,j):i\in V(F_{v}),\text{ }%
j\in V(F_{w})\text{ and }(v,w)\in A(D)\}.
\end{eqnarray*}

The composition of digraphs is not commutative, but it is associative (an
easy proof left to the reader). A digraph $D$ is said to be \textit{simple }%
(or \textit{prime}, see \cite{K-I}, p. 116) if it is not isomorphic to a
composition of digraphs (see \cite{Moon}).

\begin{proposition}[\protect\cite{VNL}, Neumann-Lara]
If $D$ is a $\overrightarrow{\omega }_{3}$-keen (resp. $\overrightarrow{%
\omega }$-keen) digraph and $F$ is a digraph, then
\begin{eqnarray*}
\overrightarrow{\omega }_{3}(D[F]) &=&\overrightarrow{\omega }_{3}(D)+%
\overrightarrow{\omega }_{3}(F)-1 \\
\text{(resp. }\overrightarrow{\omega }(D[F]) &=&\overrightarrow{\omega }(D)+%
\overrightarrow{\omega }(F)-1\text{).}
\end{eqnarray*}
Moreover, if $D$ and $F$ are $\overrightarrow{\omega }_{3}$-keen (resp. $%
\overrightarrow{\omega }$-keen) digraphs, then $D[F]$ is also $%
\overrightarrow{\omega }_{3}$-keen (resp. $\overrightarrow{\omega }$-keen). %
\label{w3-comp}
\end{proposition}

We define the interval of positive integers $[1,k]=\{1,2,...,k\}$ $(k\geq
1). $

\begin{proposition}[\protect\cite{VNL}, Neumann-Lara]
Let $\overrightarrow{C}_{2m+1}(J)$ and $\overrightarrow{C}_{2n+1}(K)$ be
circulant tournaments. Then the composition $\overrightarrow{C}_{2m+1}(J)%
\left[ \overrightarrow{C}_{2n+1}(K)\right] $ is a circulant tournament
isomorphic to $\overrightarrow{C}_{\left( 2m+1\right) \left( 2n+1\right)
}(L),$ where the symbol set \label{comp}%
\begin{equation*}
L=(2m+1)K\cup (J+(2m+1)[1,2n+1]).
\end{equation*}
\end{proposition}

Given nonempty subsets $A,B\subseteq \mathbb{Z}_{n},$ we define%
\begin{eqnarray*}
A+B &=&\left\{ a+b:a\in A,b\in B\right\} ,\text{ }cA=\{ca:a\in A,c\in
\mathbb{Z}_{n}\}, \\
-A &=&\left\{ -a:a\in A\right\} \text{ and }A-B=A+(-B).
\end{eqnarray*}

If $A$ and $B$ are sets, we denote the set difference by $A\setminus B.$ A
set $\emptyset \neq A\subseteq \mathbb{Z}_{n}$ is an \textit{arithmetic
progression} if there exist $a,d\in \mathbb{Z}_{n},$ $d\neq 0$ such that
\begin{equation*}
A=\left\{ a+id : 0\leq i\leq \left\vert A\right\vert -1\right\} .
\end{equation*}

Let $A,$ $B$ and $C$ nonempty subsets of $\mathbb{Z}_{n}.$ As a simple
consequence of the above definitions we have (see \cite{Vos})
\begin{equation}
(A+B)\cap C=\varnothing \Leftrightarrow A\cap (B-C)=\varnothing .
\label{SumInt}
\end{equation}

Let $G$ be an abelian group written additively and $\varnothing \neq
C\subseteq G.$ Following \cite{Kemp} (and \cite{Gryn}), we define the
\textit{period} of the set $C$ as%
\begin{equation*}
H(C)=\{g\in G:C+g=C\}.
\end{equation*}%
Then $C+H(C)=C$ and $H(C)$ is a subgroup of $G.$ We say that $C$ is \textit{%
periodic }if $H(C)\neq \{0\}$ and \textit{aperiodic }if $H(C)=\{0\}.$
Observe that $C$ is periodic if and only if $C$ is the union of $H(C)$%
-cosets The set $C$ is called \textit{quasi-periodic} if there exists $H,$ a
nontrivial subgroup of $G,$ such that $C=C^{\prime }\cup C^{\prime \prime }$
($C^{\prime }$ and $C^{\prime \prime }$ could be empty) and $C^{\prime }\cap
C^{\prime \prime }=\varnothing ,$ where

\begin{enumerate}
\item[(i)] $C^{\prime }$ is periodic and

\item[(ii)] $C^{\prime \prime }$ is properly contained in an $H$-coset, that
is, $C^{\prime \prime }\subset c+H,$ where $c\in C^{\prime \prime }.$
\end{enumerate}

In this setting, notice that if $C$ is periodic, then it is quasi-periodic.
It is straightforward to check that an arithmetic progression of length at
least $3$ is an aperiodic set.

\begin{theorem}[\protect\cite{Kne53} Kneser]
Let $G$ be a non-trivial abelian group and $A$ and $B$ nonempty finite
subsets of $G.$ Then $H=H(A+B)$ satisfies that
\begin{equation*}
\left\vert A+B\right\vert \geq \left\vert A+H\right\vert +\left\vert
B+H\right\vert -\left\vert H\right\vert \geq \left\vert A\right\vert
+\left\vert B\right\vert -\left\vert H\right\vert
\end{equation*}
Moreover, $A+B$ is periodic if $\left\vert A+B\right\vert \leq \left\vert
A\right\vert +\left\vert B\right\vert -2.$ \label{Kneser}
\end{theorem}

The following result is a consequence of more general theorems proved in
\cite{Kemp} by Kemperman, see also \cite{Gryn} for a thorough study of the
so called Kemperman Structure Theorem.

\begin{theorem}[\protect\cite{Kemp} Kemperman]
Let $G$ be an abelian finite group of order $n$ and $A,B$ subsets of
$G$ such that $\left\vert A\right\vert \geq 2,$ $\left\vert
B\right\vert \geq 2$ and $\left\vert A+B\right\vert =\left\vert
A\right\vert +\left\vert B\right\vert -1.$ Then either $A+B$ is an
arithmetic progression or $A+B$ is quasi-periodic with period $H$, a
nontrivial subgroup of $G.$ \label{Kemperman2}
\end{theorem}

\section{Compositions of circulant tournaments and quasi-periodic symbol sets%
}

Let $m$ be a positive integer. Observe that $\gcd (m,2m+1)=1$ and as a
consequence, if $J$ is the set symbol of a circulant tournament, then $J$ is
not periodic. Therefore, $J$ is aperiodic or quasi-periodic. Using Theorems %
\ref{Kneser} and \ref{Kemperman2} and the definition of the symbol set $J$
(recall that $0\notin J$ for the second inequality), we have that%
\begin{equation*}
2m-1=2\left\vert J\right\vert -1\leq \left\vert J+J\right\vert \leq
2\left\vert J\right\vert =2m.
\end{equation*}

The following theorem gives a characterization of compositions of circulant
tournaments in terms of the quasi-periodicity of its symbol set.

\begin{theorem}
Let $\overrightarrow{C}_{2p+1}(L)$ be a circulant tournament. $%
\overrightarrow{C}_{2p+1}(L)$ is a composition of circulant tournaments if
and only if $L$ is quasi-periodic. \label{comp-quasi}
\end{theorem}

\begin{proof}
Suppose that $\overrightarrow{C}_{2p+1}(L)$ is isomorphic to $%
\overrightarrow{C}_{2m+1}(J)\left[ \overrightarrow{C}_{2n+1}(K)\right] $.
According to Proposition \ref{comp},
\begin{equation*}
L=(2m+1)K\cup (J+(2m+1)[1,2n+1])
\end{equation*}
and $2p+1=(2m+1)(2n+1).$ First, notice that
\begin{eqnarray*}
H &=&(2m+1)[1,2n+1]=\{(2m+1)i\hspace{4pt}\textrm{mod}\left(
2p+1\right) :i=0,1,\ldots ,2n\}
\\
&=&\{0,2m+1,4m+2,\ldots ,2n(2m+1)\}
\end{eqnarray*}%
is a subgroup of $\mathbb{Z}_{2p+1}$ of order $2n+1$ and every nontrivial
subgroup of the cyclic group $\mathbb{Z}_{2p+1}$ is of this form if $2n+1$
is a divisor of $2p+1.$
We have that $J\cap H=\varnothing$, since $J\subseteq \mathbb{Z}_{2m+1}\setminus \{0\}$. Therefore, $J+H$ is a union of $H$-cosets not
containing the subgroup $H$ $(0\notin J).$ Similarly, $K\subseteq \mathbb{Z}%
_{2n+1}\setminus \{0\}$ and consequently $(2m+1)K\subset H$ (recall that all
products are taken modulo $2p+1).$ We have that $(2m+1)K\cap
(J+H)=\varnothing .$ Let $L=C^{\prime }\cup C^{\prime \prime },$ where $%
C^{\prime }=J+H$ and $C^{\prime \prime }=(2m+1)K.$ We conclude that $L$ is
quasi-periodic.

Conversely, suppose that the symbol set $L$ of cardinality $p$ is
quasi-periodic. By the definition of a quasi-periodic set, there exists $H,$ a nontrivial subgroup of $%
\mathbb{Z}_{2p+1},$ such that $L=C^{\prime }\cup C^{\prime \prime }$ and $%
C^{\prime }\cap C^{\prime \prime }=\varnothing ,$ where $C^{\prime
}+H=C^{\prime }$ and $C^{\prime \prime }\subset c+H$ with $c\in C^{\prime
\prime }$. Observe that $H\cap C^{\prime
}=\varnothing $ because $0\notin L.$ Then
\begin{equation*}
C^{\prime }=\{j_{i}+H:0\neq j_{i}\in G\}.
\end{equation*}%
Let $J=\{j_{1},j_{2},...,j_{s}\}$, where $s$ is the number of cosets contained in $%
C^{\prime}$. Therefore $C^{\prime }=J+H$. Observe that if $j_{k}\in J$ where $%
k\in \lbrack 1,s],$ then $-j_{k}\notin J$ (otherwise, $j_{k}\in L$ and $%
-j_{k}\in L,$ a contradiction). Moreover, $j_{k}+H\subseteq L$ if and only
if $-j_{k}+H\nsubseteq L$ by the definition of the symbol set $L$.
Observe that in the quotient group $\mathbb{Z}_{2p+1}/H$, we have that $-($
$j_{k}+H)=$ $-j_{k}+H.$ Consequently, $C^{\prime \prime }\subset c+H=H.$ Let
\begin{equation*}
C^{\prime \prime }=\{(2m+1)k_{\alpha }:k_{\alpha }\in \lbrack 1,2n+1]\}.
\end{equation*}%
We define $K=\{k_{1},k_{2},...,k_{t}\}.$ Once again, if $k_{\alpha }\in K,$
then $-k_{\alpha }\notin K$ for every $\alpha \in \lbrack 1,t].$ Observe
also that $\left\vert C^{\prime \prime }\right\vert =t.$ Therefore, $%
s(2n+1)+t=p.$ Since $p=m(2n+1)+n$, where $m,$ $n,$ $s$ and $t$ are positive
integers, and $2n+1>\left\vert n-t\right\vert $ (here $\left\vert
x\right\vert $ stands for the absolute value of the integer $x)$, we have
that
\begin{equation*}
s(2n+1)+t=m(2n+1)+n \Leftrightarrow
(s-m)(2n+1)=n-t\Leftrightarrow s=m\text{ and }n=t.
\end{equation*}
Hence $L=(2m+1)K\cup (J+H).$ By Proposition \ref{comp}, we conclude that
$\overrightarrow{C}_{2p+1}(L)$ is isomorphic to $\overrightarrow{C}_{2m+1}(J)%
\left[ \overrightarrow{C}_{2n+1}(K)\right] $.
\end{proof}

Using Theorem \ref{Kemperman2}, we have the following consequences.

\begin{corollary}
Let $\overrightarrow{C}_{2m+1}(J)$ be a circulant tournament such that $J$
is not an arithmetic progression. Then $\left\vert J+J\right\vert
=2\left\vert J\right\vert -1$ if and only if $J$ is quasi-periodic.
Moreover, if $J$ is an arithmetic progression (resp. quasi-periodic), then $%
J+J$ is an arithmetic progression (resp. quasi-periodic).
\end{corollary}

\begin{corollary}
Let $\overrightarrow{C}_{2m+1}(J)$ be a simple circulant tournament. Then $J$
and $J+J$ are aperiodic. \label{J-aperiodic}
\end{corollary}

\section{Simple circulant tournaments are $\protect\overrightarrow{\protect%
\omega }_{3}$-keen and tight}

Let $\overrightarrow{C}_{2n+1}(J)$ be a circulant tournament. We use the
following results proved in \cite{Ll-NL} for circulant tournaments of prime
order. The following first lemma remains valid for circulant tournaments of
any odd order (see Remark 4 of the already mentioned paper). The second one
is also true for general circulant tournaments since the proof only uses the
additive properties of the ring $\mathbb{Z}_{2n+1}.$

\begin{lemma}[\protect\cite{Ll-NL}, Lemma 1]
Let $\pi =A\mid B\mid C$ be an externally $\overrightarrow{C}_{3}$-free
partition of $\overrightarrow{C}_{2n+1}(J).$ Then \label{ABCJ}
\begin{equation*}
\left( \left( \left( \left( \left( A+J\right) \cap B\right) +J\right) \cap
C\right) +J\right) \cap A=\varnothing .
\end{equation*}
\end{lemma}

\begin{lemma}[\protect\cite{Ll-NL}, Lemma 2]
Let $A$, $C$ and $J$ be nonempty subsets of $\mathbb{Z}_{2n+1}$ such that $%
A\cap C=\varnothing $ and $\overline{J}=-J\cup \{0\}.$ Then \label{Aux}%
\begin{equation*}
\left( C\cap (A-J\right) )-J=(C-J)\cap (A-J-J).
\end{equation*}
\end{lemma}

Suppose that $\pi =A\mid B\mid C$ is an externally $\overrightarrow{C}_{3}$%
-free partition of $\overrightarrow{C}_{2n+1}(J).$ Then by Lemma \ref{ABCJ}
and using (\ref{SumInt}) we have
\begin{gather*}
\quad \left( \left( \left( \left( \left( A+J\right) \cap B\right) +J\right)
\cap C\right) +J\right) \cap A=\varnothing \\
\Leftrightarrow \left( \left( \left( \left( A+J\right) \cap B\right)
+J\right) \cap C\right) \cap \left( A-J\right) =\varnothing \\
\Leftrightarrow \left( \left( \left( A+J\right) \cap B\right) +J\right) \cap
\left( C\cap \left( A-J\right) \right) =\varnothing \\
\Leftrightarrow \left( \left( A+J\right) \cap B\right) \cap \left( \left(
C\cap \left( A-J\right) \right) -J\right) =\varnothing .
\end{gather*}%
By Lemma \ref{Aux}, the last equality is equivalent to
\begin{gather*}
\left( \left( A+J\right) \cap B\right) \cap \left( \left( C-J\right) \cap
\left( A-J-J\right) \right) =\varnothing \\
\Leftrightarrow \left( A+J\right) \cap B\cap \left( C-J\right) \subseteq
\overline{A-J-J},
\end{gather*}%
where the complement is relative to $\mathbb{Z}_{2n+1}.$ So we have proved
the following

\begin{lemma}
Let $\pi =A\mid B\mid C$ be an externally $\overrightarrow{C}_{3}$-free
partition of $\overrightarrow{C}_{2n+1}(J).$ Then \label{AJJ}%
\begin{gather*}
\left( \left( \left( \left( \left( A+J\right) \cap B\right) +J\right) \cap
C\right) +J\right) \cap A=\varnothing \\
\Leftrightarrow \left( A+J\right) \cap B\cap \left( C-J\right) \subseteq
\overline{A-J-J}.
\end{gather*}
\end{lemma}

We will follow the procedure provided in \cite{Ll-NL}. First, we
show the following proposition, whose proof is made in the same way as did
in Proposition 5 of \cite{Ll-NL}, using Lemmas \ref{ABCJ}, \ref{Aux} and \ref%
{AJJ} stated before. In fact, the proof we give here is shorter than that
given in our previous paper.

\begin{proposition}
Let $n$ be a positive integer and $\overrightarrow{C}_{2n+1}(J)$ be a
circulant tournament, where $J$ is neither an arithmetic progression nor a
quasi-periodic set. Then $\overrightarrow{C}_{2n+1}(J)$ is $\overrightarrow{%
\omega }_{3}$-keen. \label{w3-keen}
\end{proposition}

\begin{proof}
Let us suppose that $\overrightarrow{\omega }_{3}(\overrightarrow{C}%
_{2n+1}(J))=k\geq 2.$ If $k=2,$ then by Remark \ref{obs1}(ii), $%
\overrightarrow{C}_{2n+1}(J)$ is $\overrightarrow{\omega }_{3}$-keen.
Therefore, we consider the case when $k\geq 3$. By Proposition \ref{prop1},
the tournament $\overrightarrow{C}_{2n+1}(J)$ has at most one singular
equivalence class in every externally $\overrightarrow{C}_{3}$-free
partition. Suppose that there exists an optimal externally $\overrightarrow{C%
}_{3}$-free partition of $V(\overrightarrow{C}_{2n+1}(J))$ without singular
equivalence classes. Let $\pi =V_{1}\mid V_{2}\mid \cdots \mid V_{k}$ $%
(k\geq 3)$ be this optimal externally $\overrightarrow{C}_{3}$-free
partition of $\mathbb{Z}_{p}$ such that $2\leq \left\vert V_{1}\right\vert
\leq \left\vert V_{2}\right\vert \leq \cdots \leq \left\vert
V_{k}\right\vert $ and define the sets%
\begin{equation*}
A=V_{1},\text{ }B=\bigcup\limits_{i=1}^{m}V_{j_{i}}\text{ and }%
C=\bigcup\limits_{i=m}^{k}V_{j_{i}},
\end{equation*}%
where $1\leq m\leq k-1$ and the set family $\{V_{j_{i}}:i=2,3,...k\}$ is a
permutation of $\{V_{j}:j=2,3,...k\}.$ Without loss of generality, we can
assume that $2\leq \left\vert A\right\vert \leq \left\vert B\right\vert $
and $\left\vert A\right\vert \leq \left\vert C\right\vert .$ Since $\pi $ is
an externally $\overrightarrow{C}_{3}$-free partition, so is $\pi ^{\prime
}=A\mid B\mid C$. Then by Lemma \ref{ABCJ},
\begin{equation*}
\left( \left( \left( \left( \left( A+J\right) \cap B\right) +J\right) \cap
C\right) +J\right) \cap A=\varnothing
\end{equation*}%
which is equivalent to $$\left( A+J\right) \cap B\cap \left( C-J\right)
\subseteq \overline{A-J-J}$$
by Lemma \ref{AJJ} (recall that the complement
is relative to $\mathbb{Z}_{2n+1}).$ Since $J$ is neither an arithmetic
progression nor a quasi-periodic set by assumption, $J$ is aperiodic and
hence $J+J$ is aperiodic too in virtue of Corollary \ref{J-aperiodic}. Hence, $%
\left\vert J+J\right\vert =2n$ and applying Theorem \ref{Kemperman2}, we
have that
\begin{equation*}
2n+1\geq \left\vert A+(-J-J)\right\vert =\left\vert A-J-J\right\vert \geq
\left\vert A\right\vert +\left\vert -(J+J)\right\vert =\left\vert
A\right\vert +2n.
\end{equation*}%
This implies that $\left\vert A\right\vert \leq 1,$ a contradiction to the
assumption that $\left\vert A\right\vert \geq 2.$
\end{proof}

Notice that cyclic circulant tournaments $\overrightarrow{C}%
_{2n+1}(1,2,...,n)$ (the case when $J$ is an arithmetic progression) are
tight (Theorem 4.4(i) of \cite{VNL}) and therefore, $\overrightarrow{\omega }%
_{3}$-keen by Remark \ref{obs1}(ii). Applying Proposition \ref{w3-keen} to
the case when $J$ is aperiodic, we have the following

\begin{corollary}
Simple circulant tournaments $\overrightarrow{C}_{2n+1}(J)$ are $%
\overrightarrow{\omega }_{3}$-keen. \label{cor1}
\end{corollary}

On the other hand, if $J$ is quasi-periodic, then Theorem \ref{comp-quasi} implies that $\overrightarrow{C}%
_{2n+1}(J)$ is a composition of circulant tournaments, say $T_{1}[T_{2}]$ ($T_{1}$ and $T_{2}$ not necessarily distinct).
 If $T_{1}$ is not simple, then $%
T_{1}=T_{3}[T_{4}]$. By Proposition \ref{comp}, the tournaments $T_{3}$ and $T_{4}$ are circulant.
By the associativity of the composition,
\begin{equation*}
T_{1}[T_{2}]=(T_{3}[T_{4}])[T_{2}]=T_{3}[T_{4}[T_{2}]].
\end{equation*}%
Analogously, if $T_{2}$ is not simple, then there exist circulant tournaments $T_{5}$ and $T_{6}$ such that
$$T_{2}=T_{5}[T_{6}] \text{ and } T_{1}[T_{2}]=T_{1}[T_{5}[T_{6}]].$$
If both $T_{1}$ and $T_{2}$ are not simple, then we have a composition of type
$$ T_{1}[T_{2}]=(T_{3}[T_{4}])[T_{5}[T_{6}]]=T_{3}[T_{4}[T_{5}[T_{6}]]]. $$
Observe that for $3 \leq j \leq 6 $, the tournaments $T_{j}$ are not necessarily distinct.
The procedure
can be repeated until the tournament $\overrightarrow{C}_{2n+1}(J)$ is
isomorphic to the nested composition
\begin{equation*}
T_{1}[T_{2}[\cdots T_{k-1}[T_{k}]\cdots ]],
\end{equation*}%
where $T_{i}$ is simple for every $1\leq i \leq k.$ We have proved the following proposition.

\begin{proposition}
Let $\overrightarrow{C}_{2n+1}(J)$ be a composition of circulant
tournaments. Then there exist $T_{1},T_{2},\ldots ,T_{k}$ simple circulant
tournaments such that $\overrightarrow{C}_{2n+1}(J)$ is isomorphic to \label%
{decomp}%
\begin{equation*}
T_{1}[T_{2}[\cdots T_{k-1}[T_{k}]\cdots ]].
\end{equation*}
\end{proposition}

We notice that the proposition stated before is a very special case of a
decomposition of a tournament (digraph) into simple (prime) tournaments
(digraphs) in the case of the lexicographic product. The problem for general
graphs is studied in Chapter 6 of \cite{K-I} by Klav\v{z}ar and Imrich and
for digraphs in \cite{D-I} by D\"{o}rfler and Imrich.

Proposition \ref{decomp}, Corollary \ref{cor1} and the repeated application
of Proposition \ref{w3-comp} yield the following consequence:

\begin{corollary}
Compositions of circulant tournaments are $\overrightarrow{\omega }_{3}$%
-keen. \label{cor2}
\end{corollary}

Putting together Proposition \ref{w3-keen} and Corollaries \ref{cor1} and %
\ref{cor2}, we obtain the following result.

\begin{theorem}
Every circulant tournament is $\overrightarrow{\omega }_{3}$-keen. \label%
{keen-w3}
\end{theorem}

In a similar way one can prove that

\begin{theorem}
Every circulant tournament is $\overrightarrow{\omega }$-keen. \label{keen-w}
\end{theorem}

Theorems \ref{keen-w3} and \ref{keen-w} affirmatively answer the problem
posed in \cite{VNL} (p. 623, end of Section 3) already mentioned in the
Introduction and generalize Theorem 3 of \cite{Ll-NL}.

The next result is proved in a similar (not exact) way as Theorem 4 of \cite%
{Ll-NL}. The proof also uses a well-known result due to B. Alspach \cite{Als}%
: every arc of a regular tournament is contained in a directed cyclic
triangle.

\begin{theorem}
Every simple circulant tournament $\overrightarrow{C}_{2n+1}(J)$ is tight. %
\label{circ-tight}
\end{theorem}

\begin{proof}
For a contradiction, let us suppose that $\overrightarrow{C}_{2n+1}(J)$ is
simple and not tight. Then there exists an externally $\overrightarrow{C}%
_{3} $-free partition of $V(\overrightarrow{C}_{2n+1}(J))$ with three equivalence classes.
By Corollary \ref{cor1}, one of such classes is singular and since the
automorphism group of $\overrightarrow{C}_{2n+1}(J)$ is vertex-transitive,
we can assume without loss of generality that the externally $%
\overrightarrow{C}_{3}$-free partition of $V(\overrightarrow{C}_{2n+1}(J))$
is $\pi =\{0\}\mid B\mid C,$ where $\left\vert B\right\vert +\left\vert
C\right\vert =2n.$ Without loss of generality, we suppose that $\left\vert
B\right\vert \leq \left\vert C\right\vert .$ Taking $A=\{0\}$ in Lemma \ref%
{AJJ}, we have that%
\begin{equation*}
J_{B}\cap (C-J)\subseteq \overline{-J-J}=\overline{-(J+J)},
\end{equation*}%
where $J_{B}$ $=J\cap B$ (similarly, $-J_{B}=\left( -J\right) \cap B$, $%
J_{c}=J\cap C$ and $-J_{C}=\left( -J\right) \cap C)$)$.$ Since $0\notin J+J$
and $J$ is aperiodic by Corollary \ref{J-aperiodic}, it follows that
\begin{equation*}
\left\vert J+J\right\vert =2\left\vert J\right\vert =2n.
\end{equation*}%
Therefore, $0\notin -(J+J)$ and
\begin{equation*}
\left\vert -(J+J)\right\vert =2\left\vert -J\right\vert =2\left\vert
J\right\vert =2n.
\end{equation*}%
From this, $\overline{-(J+J)}=\{0\}$ and consequently $J_{B}\cap
(C-J)\subseteq \{0\}.$ Observe that if $J_{B}\cap (C-J)=\{0\},$ then $0\in J$
and $0\in B,$ which is a contradiction. Therefore, \textit{\ }$J_{B}\cap
(C-J)=\varnothing .$ By (\ref{SumInt}), this equality is equivalent to%
\begin{equation*}
(J_{B}+J)\cap C=\varnothing \Leftrightarrow J_{B}+J\subseteq \overline{C}%
=B\cup \{0\}.
\end{equation*}%
Note that $0\notin J_{B}+J\subseteq J+J$ and so, $$J_{B}+J\subseteq
B=J_{B}\cup \left( -J_{B}\right) .$$ Then
\begin{equation*}
\left\vert J_{B}\right\vert +\left\vert J\right\vert \leq \left\vert
J_{B}+J\right\vert \leq \left\vert J_{B}\right\vert +\left\vert
-J_{B}\right\vert \leq n,
\end{equation*}%
where for the inequality on the left side we use that
$\overrightarrow{C}_{p}(J)$ is simple, Theorem \ref{Kemperman2} and
Corollary \ref{J-aperiodic} and
for that on the right side, we use that $J_{B}\cap (-J_{B})=\varnothing ,$ $%
\left\vert B\right\vert +\left\vert C\right\vert =2n$ and $\left\vert
B\right\vert \leq \left\vert C\right\vert $. It follows from the last
inequality that $\left\vert J\right\vert =\left\vert -J_{B}\right\vert =n.$

Since $-J_{B}\subseteq -J$ and $\left\vert J\right\vert =\left\vert
-J\right\vert =n$, we have that $\left\vert -J_{B}\right\vert =\left\vert
-J\right\vert =n$ and so,
\begin{equation*}
\left\vert -J_{B}\right\vert =n\Leftrightarrow -J_{C}=\varnothing .
\end{equation*}%
We conclude that $-J\subseteq B$ and $C\subseteq J.$ Then, the tournament $%
\overrightarrow{C}_{2n+1}(J)$ contains the arcs $0\longrightarrow j$ for all
$j\in J_{C}=C.$ By the already mentioned result by Alspach (see\cite{Als}),
there exists a directed cyclic triangle
\begin{equation*}
\begin{array}{ccccc}
&  & 0 &  &  \\
& \swarrow &  & \nwarrow &  \\
J\supseteq C\ni j &  & \longrightarrow &  & j^{\prime }\in -J\subseteq B%
\end{array}%
.
\end{equation*}%
We have obtained a contradiction, $\pi =\{0\}\mid B\mid C$ is an externally $%
\overrightarrow{C}_{3}$-free partition of $V(\overrightarrow{C}_{2n+1}(J)).$
\end{proof}

\begin{corollary}[\protect\cite{Ll-NL}, Theorem 4]
Let $p$ be an odd prime number. Then $\overrightarrow{C}_{p}(J)$ is tight.
\end{corollary}

Using Proposition \ref{w3-comp} and Theorems \ref{comp-quasi} and \ref%
{circ-tight}, we can state the following characterization of simple
circulant tournaments (and therefore, of compositions of circulant
tournaments).

\begin{corollary}
Let $\overrightarrow{C}_{2n+1}(J)$ be a circulant tournament. The following
conditions are equivalent:

\begin{enumerate}
\item[(i)] $\overrightarrow{C}_{2n+1}(J)$ is tight.

\item[(ii)] $\overrightarrow{C}_{2n+1}(J)$ is simple.

\item[(iii)] $J$ is aperiodic. \label{char}
\end{enumerate}
\end{corollary}

From the definitions of the acyclic and the $\overrightarrow{C}_{3}$-free
disconnection, we have that $2\leq \overrightarrow{\omega }(T)\leq
\overrightarrow{\omega }_{3}(T)$ for every tournament $T$. Using again
Theorems \ref{keen-w3} and \ref{keen-w} combined with Proposition \ref%
{w3-comp}, we obtain the following corollary which solves Problem 6.6.3 of
\cite{VNL}.

\begin{corollary}
$\overrightarrow{\omega }(\overrightarrow{C}_{2n+1}(J))=\overrightarrow{%
\omega }_{3}(\overrightarrow{C}_{2n+1}(J))$ for every positive integer $n.$ %
\label{final}
\end{corollary}


\end{document}